\documentclass[sn-mathphys-num,final]{sn-jnl}

\newcommand{\displayfigures}{true}

\usepackage{comment}
\usepackage{graphicx}%
\usepackage{multirow}%
\usepackage{amsmath,amssymb,amsfonts}%
\usepackage{amsthm}%
\usepackage{mathrsfs}%
\usepackage[title]{appendix}%
\usepackage{xcolor}%
\usepackage{textcomp}%
\usepackage{manyfoot}%
\usepackage{booktabs}%
\usepackage{algorithm}%
\usepackage{algorithmicx}%
\usepackage{algpseudocode}%
\usepackage{listings}%

\allowdisplaybreaks
\usepackage{mathtools}
\usepackage{svg}

\newcommand{\R}{\mathbb{R}}

\newtheorem{conjecture}{Conjecture}

\newcommand{\Par}[1]{\left( #1 \right)}

\newcommand{\GramTwo}[2]{\mathrm{Gram}\vrow{#1}{#2}}
\newcommand{\GramFour}[4]{\mathrm{Gram}\begin{pmatrix} #1 & #2 & #3 & #4 \end{pmatrix}}

\newcommand{\Gramm}[4]{\mathrm{Gram}\comN{2}\left(#1,#2,#3,#4\right)}

\newcommand\vcol[2]{\begin{pmatrix} #1 \\ #2 \end{pmatrix}}
\newcommand\vrow[2]{\begin{pmatrix} #1 & #2 \end{pmatrix}}
\newcommand\vrowFour[4]{\begin{pmatrix} #1 & #2 & #3 & #4 \end{pmatrix}}
\newcommand{\vmat}[4]{\begin{pmatrix} #1 & #2 \\ #3 & #4 \end{pmatrix}}

\newcommand{\MS}[1]{ \{\hspace{-0.1cm}\{ #1 \}\hspace{-0.1cm}\}}

\newcommand\GM[1]{\mathcal{G}_{#1}}
\newcommand\GMM[1]{\mathcal{P}_{#1}}
\newcommand{\LL}[1]{\mathcal{L}_{#1}}
\newcommand{\Ltal}[1]{\mathcal{L}^{\mathrm{tall}}_{#1}}
\newcommand{\Lwide}[1]{\mathcal{L}^{\mathrm{wide}}_{#1}}
\newcommand\Utal{\mathcal{L}_{\{1\}}^{\mathrm{tall}}}

\newcommand{\SSS}[1]{\mathcal{S}_{#1}}

\usepackage{tikz}

\newcommand{\figureH}[1]{%
    \ifthenelse{\equal{\displayfigures}{false}}
      {
       \textcolor{blue}{\textit{hidden}}
      }
      {
      #1
      }
}

\usepackage[commandnameprefix=always,todonotes={textsize=tiny},authormarkup=none]{changes} 
\definecolor{DZCcolor}{HTML}{074799}
\definecolor{DZCcolorr}{HTML}{974799}
\definechangesauthor[name={Zhicheng}, color=DZCcolorr]{D}
\definechangesauthor[name={Zhicheng}, color=DZCcolorr]{Da}
\definechangesauthor[name={Nizar}, color=blue]{N}
\newcommand{\comN}[1]{#1}
\newcommand{\comD}[1]{#1}

\newcommand{\addDJune}[2][]{#2} 
\newcommand{\comDJune}[2][]{#2} 

\newcommand{\repDAug}[3][]{#2} 
\newcommand{\repDOct}[3][]{\chreplaced[id=D,comment={#1}]{#2}{#3}}
\newcommand{\repDOctui}[3][]{#2}
\newcommand{\deleD}{\chdeleted[id=D]}
\newcommand{\addD}{\chadded[id=D]}%

\theoremstyle{thmstyleone}%
\newtheorem{theorem}{Theorem}[section]
\newtheorem{proposition}[theorem]{Proposition}%
\newtheorem{lemma}[theorem]{Lemma}%
\newtheorem{corollary}[theorem]{Corollary}%

\theoremstyle{thmstyletwo}%
\newtheorem{remark}{Remark}%

\theoremstyle{thmstylethree}%
\newtheorem{definition}{Definition}[section]%

\begin{document}

\title[Article Title]{On the Convex Interpolation for Linear Operators}


\author*[1]{\fnm{Nizar} \sur{Bousselmi}}\email{nizar.bousselmi@uclouvain.be}
\equalcont{These authors contributed equally to this work.}

\author*[2]{\fnm{Zhicheng} \sur{Deng}}\email{dengzhch@shanghaitech.edu.cn}
\equalcont{These authors contributed equally to this work.}

\author[2]{\fnm{Jie} \sur{Lu}}\email{lujie@shanghaitech.edu.cn}

\author[1,3]{\fnm{François} \sur{Glineur}}\email{francois.glineur@uclouvain.be}

\author[1]{\fnm{Julien M.} \sur{Hendrickx}}\email{julien.hendrickx@uclouvain.be}

\affil[1]{\orgdiv{ICTEAM}, \orgname{UCLouvain}, \orgaddress{\city{Louvain-la-Neuve}, \country{Belgium}}}

\affil[2]{\orgdiv{School of Information Science and Technology}, \orgname{Shanghaitech University}, \orgaddress{\city{Shanghai}, \country{P. R. China}}}

\affil[3]{\orgdiv{CORE}, \orgname{UCLouvain}, \orgaddress{\city{Louvain-la-Neuve}, \country{Belgium}}}


\abstract{%
\comD{The worst-case performance of an optimization method on a problem class \comN{can be} analyzed using a finite description of the problem class, {known as} interpolation conditions.} In this work, we study interpolation conditions for linear operators given scalar products between discrete inputs and outputs.
First, we show that if only convex constraints on the scalar products of inputs and outputs are allowed, it is only possible to characterize classes of linear operators or symmetric linear operator whose all singular values or eigenvalues belong to some subset of $\R$. Then, we propose new interpolation conditions for linear operators with minimal and maximal singular values and linear operators whose eigenvalues or singular values belong to unions of subsets. Finally,  we illustrate the new interpolation conditions through the analysis of the Gradient and Chambolle-Pock methods. It allows to obtain new numerical worst-case guarantees on these methods.
}


\maketitle

\section{Introduction}\label{sect:intro}
\comN{One recent way to analyze} the worst-case performance of a given optimization method on a given function class is the well-developed Performance Estimation Problem (PEP) framework \cite{drori2014performance,taylor2017smooth}. It formulates the problem of computing the worst-case performance of the method on the class as an optimization problem and then solves it.
The framework relies on interpolation conditions of the function classes considered \cite{rubbens2024constraint,goujaud2023fundamental} \comN{(such interpolation conditions can serve larger purposes)}. Worst-case guarantees are built up by combining inequalities between specific points of the domain characterizing the function class. In some sense, interpolation conditions of a function class are the best description of this class.

Analyzing the performance of optimization methods on problems of the form $\min_x g(Mx)$ or $\min_x f(x) + g(Mx)$ where $f\in \mathcal{F}$, $g\in \mathcal{G}$, and $M \in \mathcal{M}$ with PEP requires interpolation conditions for the function classes $\mathcal{F}$ and $\mathcal{G}$ and linear operator class $\mathcal{M}$ (see \cite{taylor2017smooth,bousselmi2024interpolation} for more details). We will focus on the interpolation of linear operators defined as follows.
\begin{definition}[Operator interpolability]\label{def:R_mat_inter1}
Given a class of linear operators $\mathcal{M}$, the sets of pairs $\{(x_i,y_i)\}_{i\in [N_1]}$ and $\{(u_j,v_j)\}_{j\in [N_2]}$ are $\mathcal{M}$-interpolable if
\begin{equation}
\exists M\in \mathcal{M}~:
\begin{cases}
y_i = M x_i,   & \forall i\in [N_1] \\
v_j = M^T u_j, & \forall j\in [N_2]
\end{cases}
\end{equation}
\end{definition}
\repDOct[
]{%
While given sets of points $\{(x_i,y_i)\}_{i\in [N_1]}$ and $\{(u_j,v_j)\}_{j\in [N_2]}$, this question can often easily be solved by computing $M$ explicitly, the PEP methodology requires an explicit description of the whose sets of points for which the desired interpolable property holds.}{}
Interpolation conditions are \comN{necessary and sufficient} constraints on the set of points $\{(x_i,y_i)\}_{i\in [N_1]}$ and $\{(u_j,v_j)\}_{j\in [N_2]}$ guaranteeing their interpolability. A previous work \cite{bousselmi2024interpolation}, developed interpolation conditions for the classes of linear operators with singular values in $[0,L]$, $\LL{[0,L]}$, and for symmetric linear operators with eigenvalues in $[\mu,L]$, $\mathcal{S}_{[\mu,L]}$, with $\mu \leq L$. Formally, given subset $S\subseteq \R$,
we denote
\begin{align}
\LL{S}        & = \{ M : \sigma_{i}(M) \in S,~\forall i \} \label{eq:sigma_in_S}                        \\
\mathcal{S}_S & = \{ M \text{ is symmetric }: \lambda_{i}(M) \in S,~\forall i \} \label{eq:lambda_in_S} 
\end{align}
where $\sigma_{i}(M)$ and $\lambda_i(M)$ are the singular values and eigenvalues of $M$.

Given the sets of pairs $\{(x_i,y_i)\}_{i\in [N_1]}$ and $\{(u_j,v_j)\}_{j\in [N_2]}$ where $x_i, u_j \in \R^n , v_j, y_i \in \R^m$, we denote
$X = (x_1~\cdots~x_{N_1})$, $Y = (y_1 ~\cdots ~y_{N_1})$, $U = (u_1 ~\cdots ~u_{N_2})$, $V = (v_1 ~\cdots ~v_{N_2})$.
Interpolation conditions for $\LL{[0,L]}$ and $\mathcal{S}_{[\mu,L]}$-interpolability are the following \cite{bousselmi2024interpolation} \comN{(there also exist interpolation conditions for skew-symmetric linear operators \cite[Corollary 3.2]{bousselmi2024interpolation})}.
\begin{theorem}[{$\LL{[0,L]}$}-interpolation conditions, Theorem 3.1 of \cite{bousselmi2024interpolation}]\label{th:int_cond_non_sym}
\comN{Given $L\geq 0$.} The sets of pairs $\{(x_i,y_i)\}_{i\in [N_1]}$ and $\{(u_j,v_j)\}_{j\in [N_2]}$ are {$\LL{[0,L]}$}-interpolable if, and only if,
\begin{align} \label{cond:non_symm}
\begin{split}
\begin{cases}
X^T V = Y^TU,          \\
Y^TY \preceq L^2 X^TX, \\
V^TV \preceq L^2 \: U^TU.
\end{cases}
\end{split}
\end{align}
\end{theorem}

\begin{theorem}[{$\mathcal{S}_{[\mu,L]}$}-interpolation conditions, Theorem 3.3 of \cite{bousselmi2024interpolation}]\label{th:int_cond_sym}
Given $\mu\leq L$. The set $\{(x_i,y_i)\}_{i\in [N_1]}$ is {$\SSS{[\mu,L]}$}-interpolable if, and only if,
\begin{align} \label{cond:symm}
\begin{split}
\begin{cases}
X^T Y = Y^T X, \\
(Y-\mu X)^T(LX-Y) \succeq 0.
\end{cases}
\end{split}
\end{align}
\end{theorem}
Thanks to these interpolation conditions and PEP, new worst-case guarantees were obtained on the Gradient and Chambolle-Pock methods \cite[Section 4]{bousselmi2024interpolation}, Condat-V\~u methods \cite{bousselmi2024comparison}, decentralized optimization \cite{colla2024computer}, min-max optimization \cite{vasin2025solving,krivchenko2024convex} and parametric quadratic optimization \cite{ranjan2024verification}.

\paragraph{Worst-case analysis and Gram representation}
Both classical worst-case analysis of optimization methods and PEP require algebraic characterizations of function classes \cite{nesterov2018lectures,rubbens2023interpolation}, to obtain \repDOct[]{guaranteed}{} tight and unimprovable bounds.\repDOctui[JH: 1. it's a bit strange to explain this once we already introduced the interpolation conditions before. It seems a bit redundant with the explanation prior to the interpolation conditions. 2. this comment is a bit too strong. There are case were non-tight conditions work. The real statement is that for any setting, if we do not use interpolation conditions, there is a problem for which we won't obtain a tight bound. However this is too complex for this introduction.
So maybe we should just say "to obtained guaranteed tight and unimprovable bounds"]{}{, we must use necessary and sufficient interpolation conditions.}
In this work, we focus on the PEP methodology. It appears that in a lot of cases, the problem of computing the worst-case performance of a given method on a given class of functions can be formulated as a convex problem \cite{taylor2017smooth}. A necessary condition for the PEP to be efficiently exploitable is that the interpolation conditions must be convex on the scalar products between points. We discuss interpolation conditions for linear operators. This convexity must hold on the elements of the following pair of Gram matrices when considering classes of general linear operators:
\begin{equation}\label{eq:notation_Gram}
  (G,H) = \Par{\GramTwo{X}{V}, \GramTwo{Y}{U}} \triangleq \Par{\vrow{X}{V}^{T} \vrow{X}{V}, \vrow{Y}{U}^{T} \vrow{Y}{U}}
\end{equation}
and the following Gram matrix when considering classes of (square) symmetric linear operators
\begin{equation} \label{eq:notation_Gramsquare}
G = \GramFour XYUV \triangleq \begin{pmatrix} X & Y & U & V\end{pmatrix}^{T} \begin{pmatrix} X & Y & U & V\end{pmatrix}
\end{equation}
\repDOct[]{\textit{i.e.}, all the scalar products whose existence is compatible with the dimensions.}{}
For instance, Theorem \ref{th:int_cond_non_sym} (resp.\@ \ref{th:int_cond_sym}) can be written as convex semi-definite constraints on $G$ and $H$ (resp.\@ $G$).

In this work, we focus on interpolation conditions which are convex on the above Gram matrices $G$ and $H$.
We tackle the following two questions related to operator interpolability:
\begin{itemize}
 \item[1.] Which classes of linear operators can be represented by convex interpolation conditions on \comN{products between the points of the sets} $\{(x_i,y_i)\}_{i\in [N_1]}$ and $\{(u_j,v_j)\}_{j\in [N_2]}$?
 \item[2.] Is there a tractable form of interpolation conditions for $\LL{[\mu,L]}$ when $\mu > 0$?
\end{itemize}
We here propose answers to these questions.

\paragraph{Contributions and outline of the paper}
In Section \ref{sect:limitation}, we show that \comDJune{when restricted to convex constraints on the Gram matrices of the inputs and outputs,} we can only characterize classes of linear and symmetric linear operators of the forms \eqref{eq:sigma_in_S} and \eqref{eq:lambda_in_S} (i.e., classes of linear operators for which the singular values or eigenvalues belong to a given set)  (Theorem \ref{th:limit}).
For example, it is impossible to characterize the class of symmetric linear operators for which the sum of singular values is equal to a given number \comN{with convex constraints on the scalar products}.
In Section \ref{sect:interpolation}, we present interpolation conditions for the class of linear operators $\LL{[\mu,L]}$ (Theorem \ref{th:int_cond_non_sym}). To obtain them, we rely on the polar decomposition that allows to express any linear operator as a product of a unitary and a symmetric linear operator for which we already knew interpolation conditions. In Section \ref{sect:union}, we propose interpolation conditions for linear operators whose singular values or eigenvalues belong to a union of subsets, \addDJune{extending thus Theorem~\ref{th:int_cond_sym} and \ref{th:int_cond_non_sym}, and approaching the limitations derived in Section~\ref{sect:limitation}}. In Section \ref{sect:applications}, we illustrate the application of our new interpolation conditions to analyze the worst-case behavior of Gradient and Chambolle-Pock methods \cite{chambolle2011first} with the PEP framework. We provide new numerical worst-case guarantees.

\paragraph{Notations and conventions}
\repDOct[]{%
In this paper, we denote by $\mathcal{L}$ and $\mathcal{S}$ the sets of all linear operators and symmetric linear operators of any \comN{finite} dimensions. For a nonnegative integer $N$, let $[N]$ denote the set $\{1, \ldots, N\}$.
We adopt the following conventions. Let $X,Y,U,V$ be matrices with columns $x_i,y_i,u_j,v_j$ respectively. In the general (rectangular) case, $x_i, v_i$ lie in one space and $y_i, u_i$ lie in a different space, therefore, we have access to the pair of Gram matrices \eqref{eq:notation_Gram}.
In the (square) symmetric case, all vectors $x_i,y_i,u_j,v_j$ lie in the same space, therefore, we have access to the Gram matrix \eqref{eq:notation_Gramsquare}.
}{}

\section{Limitation of representation of linear operators through Gram matrices} \label{sect:limitation}
In this section, we present a limitation \addDJune{of representation} on classes of linear operators \comDJune{when restricted to convex constraints on the Gram matrices of inputs and outputs} \repDOct[]{of the linear operators.}{}
\addDJune{We first introduce some key concepts in the Gram matrices associated with the linear operators in Section~\ref{ssec:gram_matrx}, followed by our main result on the limitation of Gram representation in Section~\ref{ssec:lim_main_result}. The relationship between the Gram representation and the interpolation condition is clarified in Section~\ref{ssec:gram_inter}.}

\subsection{Gram matrices associated with classes of linear operators}
\label{ssec:gram_matrx}
\repDOctui[Removed according to Julien's comment]{}{We define the concept of Gram matrices associated with a linear operator and with a class of linear operators.}

\begin{definition}[Gram matrix associated with a linear operator]\label{def:associated}
Given a matrix $G$ (resp.\@ a pair of Gram matrices $(G,H)$), we say that $G$ (resp.\@ $(G,H)$) is associated with $M$ if
\begin{equation}
\exists X,Y,U,V :
\begin{cases}
(G,H) = \Par{\GramTwo XV, \GramTwo YU} \text{ (resp.\@ } \GramFour XYUV) \\
Y = MX                                                                   \\
V = M^{T} U.
\end{cases}
\end{equation}
\end{definition}
\begin{definition}[\comN{Set of Gram matrices associated with a class of linear operator}]
Let classes of linear operators $\mathcal{M}\subseteq \mathcal{L}$ and square linear operators $\mathcal{Q} \subseteq \mathcal{S}$, the set of Gram matrices associated with $\mathcal{M}$ and $\mathcal{Q}$ are
\begin{align*}
\GMM{\mathcal{M}}^N & \triangleq \{ (G,H) = \Par{\GramTwo{X}{V}, \GramTwo{Y}{U}} & : \exists  M \in \mathcal{M}:  Y=MX,~V=M^{T}U  \}, \\
\GM{\mathcal{Q}}^N  & \triangleq \{ G = \GramFour XYUV                           & : \exists M \in \mathcal{Q}: Y=MX,~V=M^{T} U \}.
\end{align*}
All the Gram matrices of a set have the same dimension, i.e., $(G,H) \in \R^{2N \times 2N} \times \R^{2N \times 2N}$ (resp.\@ $G\in \R^{4N \times 4N}$). We often omit the superscript $N$ in the notations $\GMM{\mathcal{M}}^N$ and $\GM{\mathcal{Q}}^N$.

\end{definition}

A central property of Gram matrices is their unitary invariance, namely, rotating the vectors composing the Gram matrices does not modify them. Formally, given any set of vectors $A=(a_1~\cdots~a_n)$, we have
\begin{equation}
\label{eq:unit_invar}
\mathrm{Gram}(A) = \mathrm{Gram}(RA)
\end{equation}
for all unitary transformations $R$. Therefore,
Gram matrices associated with a linear operator (resp.\@ symmetric linear operator) $M$ are also associated with all linear operators of the form $RMS^{T}$ (resp.\@ $ R^{T} M R$).
\begin{lemma}[Unitary invariance of Gram matrices]\label{lem:unitary}
 Let $R$ and $S$ be any unitary transformations.
\begin{itemize}
 \item[1.] If a Gram matrix $G$ is associated with a symmetric linear operator $M$ then it is also associated with any linear operator $R M R^{T}$.
 \label{lem:unitary_symmetric}
 \item[2.] \label{lem:unitary_non_symmetric}
       If a pair of Gram matrices $(G,H)$ is associated with a linear operator $M$ then $(G,H)$ is associated with any linear operator $RMS^{T}$.
\end{itemize}
\begin{proof}
For statement 1, we have
\begin{align}\label{gram_io_diff}
G & = \GramFour{X}{Y}{U}{V}  \overset{\eqref{eq:unit_invar}}{=} \mathrm{Gram}(R \begin{pmatrix} X & Y & U & V \end{pmatrix})  = \GramFour{X'}{Y'}{U'}{V'}
\end{align}
where $Y'= R M R^{T} X'$, $V' = R^{T} M^{T} R U'$ for any unitary transformations $R$.
And for statement 2, we have
\begin{equation}
\begin{cases}
G = \GramTwo{X}{V} \overset{\eqref{eq:unit_invar}}{=} \mathrm{Gram}\left(S \begin{pmatrix} X &  V \end{pmatrix}\right)  = \GramTwo{X'}{V'} \\
H = \GramTwo{Y}{U} \overset{\eqref{eq:unit_invar}}{=} \mathrm{Gram}\left(R \begin{pmatrix} Y &  U \end{pmatrix}\right)  = \GramTwo{Y'}{U'}
\end{cases}
\end{equation}
where $Y'= R M S^{T} X'$, $V' = S M^{T} R^{T} U'$ for any unitary transformations $R$ and $S$.
\end{proof}
\end{lemma}
Therefore,
Gram matrices associated with a general (resp.\@ symmetric) linear operator $M$ are also associated with the diagonalized version $M' = R M S^{T}$ (resp.\@ $ M' = RMR^{T}$) whose diagonal elements are the singular values of $M$ (resp.\@ {eigenvalues} of {symmetric linear operator} $M$) in any order.
This implies that \comN{the Gram representation only allows} characterizing the spectrum of linear operators $M$, namely, the values and multiplicity of each of their singular values (or eigenvalues for symmetric linear operators), in a sense that will be made clear in Corollary \ref{cor:spectr_chara}.

\comN{We define spectrums using multiset which is a set where each element has a multiplicity. We use the usual inclusion, sum, and difference of multisets.}
\begin{definition}[Spectrum of matrices]
Let a matrix $M\in \mathcal{L}$ and a square matrix $Q \subseteq \mathcal{S}$. The singular value spectrum of $M$ is the following multiset of $\mathbb{R}$
\begin{equation}
\sigma(M) \triangleq  \MS{\text{singular values of } M }
\end{equation}
and, the eigenvalue spectrum of $Q$ is the following multiset of $\mathbb{C}$
\begin{equation}
\lambda(Q) \triangleq  \MS{\text{eigenvalues of } Q }.
\end{equation}
\end{definition}
\comN{We will typically use $\lambda$ to denote an arbitrary multiset of complex or real values with finitely many elements (meant to represent spectrums), and $\Lambda$ to denote a set of spectrums $\lambda$.}
We denote the set of \comN{(eigenvalue or singular value)} spectrums whose elements belong to a given subset $S$ by
\begin{align}
\Delta(S) & \triangleq \{ \lambda : \forall a\in \lambda, a\in S \}. \label{eq:LambdaS}
\end{align}
$\lambda$ is a multiset and considers the multiplicity of singular values and eigenvalues. We use a multiset and not a sequence since matrices with the same singular values or eigenvalues but not in the same order are the same up to a unitary transformation (and have thus the same Gram matrix by Lemma \ref{lem:unitary}).

We can now define Gram matrices associated with spectrums. More precisely, given a set of spectrums $\Lambda$, the set of Gram matrices associated with $\Lambda$ is
\begin{align}
\GMM{\Lambda} & \triangleq \left\{ (G,H) =(\GramTwo{X}{V},\GramTwo{Y}{U})\,:\,\exists M :  Y=MX, V = M^{T} U, \text{ and } \sigma(M) \in \Lambda \right \} \label{eq:Plambda}
\end{align}
and similarly, for symmetric linear operators, we have
\begin{align}
\mathcal{G}_\Lambda & \triangleq \left\{ G=\GramFour{X}{Y}{U}{V}\,:\, \exists M \text{ \comD{symmetric}}:  Y=MX, V = M^{T} U, \text{ and } \lambda(M) \in \Lambda \right \}  \label{eq:Glambda} 
\end{align}

Due to Lemma \ref{lem:unitary} and the following corollary, we focus on spaces of Gram matrices associated with a set of singular or eigenvalue spectrum $\Lambda$.

\begin{corollary}[Spectrum characterization]\label{cor:spectr_chara}
Let $\mathcal{M}$ be a set of linear operators (resp.\@ symmetric linear operators) of any \comN{finite} dimension.
We have $\GMM{\mathcal{M}} = \GMM{\Lambda}$ (resp.\@ $\GM{\mathcal{M}} = \GM{\Lambda}$) where $\Lambda = \{ \sigma(M) \,:\, M \in \mathcal{M} \}$ (resp.\@ $\Lambda = \{ \lambda(M) \,:\, M \in \mathcal{M} \}$).
\begin{proof}
Thanks to Lemma \ref{lem:unitary}, \comN{a pair of Gram matrix $(G,H)$ (resp.\@ $G$) is also associated with its \addDJune{orthogonally} diagonalized version and reciprocally, therefore,} $(G,H) \in \GMM{\mathcal{M}} \Leftrightarrow (G,H) \in \GMM{\Lambda}$ (resp.\@ $G \in \GM{\mathcal{M}} \Leftrightarrow G \in \GM{\Lambda}$).
\end{proof}
\end{corollary}

\subsection{Main result}
\label{ssec:lim_main_result}
We can now show the two main results of the section, namely, (i) sets of Gram matrices, $\GMM{\Delta(S)}$ and $\GM{\Delta(S)}$ associated with linear operators with all singular values in a subset $S \subseteq \R$ or with symmetric linear operators with all eigenvalues in a subset $S\subseteq \R$ are convex, and reciprocally, (ii)  if $\GMM{\mathcal{M}}$ or $\GM{\mathcal{M}}$ are convex, then they are associated with classes of linear operators and symmetric linear operators with singular values in some subset $S \subseteq \R$ or eigenvalues in some subset $S \subseteq \R$.
\begin{theorem}\label{lem:cone}
Let $S\subseteq \R$. $\GMM{\Delta(S)}$ (resp.\@ $\GM{\Delta(S)}$)
is convex.
\begin{proof}
We show that $\GMM{\Delta(S)}$ (resp.\@ $\GM{\Delta(S)}$) are convex cones.

Let $G_i \in \GM{\Delta(S)}$ (resp.\@ $(G_i,H_i) \in \GMM{\Delta(S)}$) for $i=1,2$. We have $G_i = \GramFour{X_i}{Y_i}{U_i}{V_i}$ (resp.\@ $(G_i,H_i) = (\GramTwo{X_i}{V_i},\GramTwo{Y_i}{U_i})$) with $Y_i = M_i X_i, V_i = M_i^{T} U_i$ and $\lambda(M_i) \in \Delta(S)$ (resp.\@ $\sigma(M_i) \in \Delta(S)$).
Let $(X~Y~U~V) \triangleq
\begin{pmatrix}
\sqrt{c_1}(X_1 & Y_1 & U_1 & V_1) \\
\sqrt{c_2}(X_2 & Y_2 & U_2 & V_2)
\end{pmatrix}$
and $M \triangleq \vmat{M_1}{0}{0}{M_2}$.
One can directly verify that $c_1 G_1 + c_2 G_2 = \GramFour{X}{Y}{U}{V}$ (resp.\@ $c_1 (G_1,H_1) + c_2 (G_2,H_2) = (\GramTwo{X}{V}, \GramTwo{Y}{U}$), $Y = M X$, $V = M^{T} U$, $\lambda(M) \in \Delta(S)$ (resp.\@ $\sigma(M) \in \Delta(S)$), and therefore $c_1 G_1 + c_2 G_2\in \GM{\Delta(S)}$ (resp.\@ $c_1 (G_1,H_1) + c_2 (G_2,H_2)\in \GMM{\Delta(S)}$).
\end{proof}
\end{theorem}

\begin{theorem}\label{th:limit}
Let $\mathcal{M}$ be a set of linear operators {(resp.\@ symmetric linear operators)}. If $\GMM{\mathcal{M}}$ (resp.\@ $\GM{\mathcal{M}}$)
is convex, then
\begin{equation}
\GMM{\mathcal{M}} = \GMM{\Delta(S)} \text{ (resp.\@ $\GM{\mathcal{M}} = \GM{\Delta(S)}$)}
\end{equation}
for some $S\subseteq \R$. 
\begin{proof}
The proof is deferred to Appendix \ref{app:proof_limit}.  
\end{proof}
\end{theorem}
Observe that it is not possible to control the multiplicity of the singular values or eigenvalues contained in the spectrums of linear operators represented by $\GMM{\Delta(S)}$ and $\GM{\Delta(S)}$.
\repDAug[]{%
\begin{remark}
When $M$ is symmetric, $\GMM{\mathcal{M}}$ is a projection of $\GM{\mathcal{M}}$, which can be seen from \eqref{eq:notation_Gram} and \eqref{eq:notation_Gramsquare}.
    If $\GM{\mathcal{M}}$ is a convex set, it follows from Theorem \ref{th:limit} that $\GM{\mathcal{M}} = \GM{\Delta(S)}$ for some $S$. Consequently, $\GMM{\mathcal{M}}$, being a projection of $\GM{\mathcal{M}}$, is convex, and therefore, $\GMM{\mathcal{M}}=\GMM{\Delta(T)}$ for some $T$.
     \repDOct[]{Recall that $\GM{\Delta{(S)}}$ is only characterized by the \repDOct[]{locations of }{} eigenvalues whereas $\GMM{\Delta{(T)}}$ is only characterized by the \repDOct[]{locations of }{} singular values. No inconsistency arises}{} since symmetric matrices have the singular values that are the absolute values of their eigenvalues. Therefore, $T = \{|\lambda| : \lambda \in S\}$.
    Finally, note that $\GMM{\Delta{(T)}}$ is associated with a larger set of linear operators than $\GM{\Delta{(S)}}$. This is natural since $P_\mathcal{M}$ is merely a projection of $G_{\mathcal{M}}$.
\end{remark}
}{}

\begin{remark}
In this section, we considered Gram matrices constructed with vectors involved by the linear operators. However, we should have also considered other vectors independent of the linear operator to be complete. In other words, instead of \eqref{eq:notation_Gramsquare} and \eqref{eq:notation_Gram}, we should use
\begin{equation}\label{extended_gram}
{\bar{G}} = \mathrm{Gram}\begin{pmatrix}
X & Y & U & V & W
\end{pmatrix}
\end{equation}
and
\begin{equation}
{(\bar{G},\bar{H})} = (\mathrm{Gram}\begin{pmatrix}
X & V & W
\end{pmatrix}, \begin{pmatrix}
Y & U & Z
\end{pmatrix})
\end{equation}
where $Y=MX$ and $V=M^{T} U$ for some linear operator $M$.
{Note that since $G$ in \eqref{eq:notation_Gramsquare} (resp.\@ $(G,H)$ in \eqref{eq:notation_Gram}) is a projection of $\bar{G}$ (resp.\@ $(\bar{G},\bar{H})$), if
\begin{align*}
\bar{\mathcal{P}}_{\mathcal{M}} & \triangleq \{
(\bar{G},\bar{H}) = (\mathrm{Gram}\begin{pmatrix}
                                  X & V & W
                                  \end{pmatrix}, \begin{pmatrix}
                                                 Y & U & Z
                                                 \end{pmatrix})
                                & : \exists  M \in \mathcal{M}, W, Z:  Y=MX,~V=M^{T}U  \},                       \\
\bar{\mathcal{G}}_{\mathcal{Q}} & \triangleq \{ \bar{G} = \mathrm{Gram}\begin{pmatrix} X & Y & U & V & W \end{pmatrix}
                                & : \exists M \in \mathcal{Q}, W: Y=MX,~V=M^{T} U \}.
\end{align*}
is convex, their projections $\GMM{\mathcal{M}}$ and $\GM{\mathcal{Q}}$ are convex.
}
All the limitation results also hold for extended Gram matrices.
\end{remark}

\subsection{Set of Gram matrices and interpolation conditions}
\label{ssec:gram_inter}
\addDJune{%
  In this section, we examine the interpolation conditions for the input-output quadruple $(X,Y,U,V)$ and how they relate to the main limitation result concerning Gram representation. \\
  Although the quadruple contains more information than its Gram matrix $G$ (or \@ the pair $(G,H)$), 
  \repDOct[
]{%
if one focus only on the locations of eigenvalues or singular values of the linear operators to be interpolated, ignoring multiplicities, then interpolation conditions for $(X,Y,U,V)$ can be derived exclusively form their Gram representations $G$ (or \@ $(G,H)$).}{%
}
}
\begin{theorem}\label{th:union_XY}
Let $S \subseteq \R$.
$(X,Y,U,V)$ is $\mathcal{S}_S$-interpolable (resp.\@ $\mathcal{L}_S$-interpolable) if, and only if, $\GramFour{X}{Y}{U}{V} \in \GM{\Delta(S)}$ (resp.\@ $(\GramTwo{X}{V}, \GramTwo{Y}{V})\in \GMM{\Delta(S)}$).
\begin{proof}
  \noindent \textit{(Necessity)} Let $(X,Y,U,V)$ being $\GM{\Delta(S)}$-interpolable, then $\exists M$ such that \comD{$\lambda(M) \in \Delta(S)$,} $Y=MX$, $V=M^{T} U$ and $G=\GramFour{X}{Y}{U}{V} \in \GM{\Delta(S)}$.\\
  \noindent \textit{(Sufficiency)} Let $G \in \GM{\Delta(S)}$, (resp.\@ $(G,H) \in \GMM{\Delta(S)}$) then there exists $X'$, $Y'$, $U'$, $V'$, $M'$ such that
\[
  \begin{aligned}
    Y'=M'X', V'=(M')^{T} U' ,\ &  G = \GramFour{X'}{Y'}{U'}{V'} \text{ and } \lambda(M') \in \Delta(S).
    \\ \text{(resp.\@ }& G = \GramTwo{X'}{V'}, H=\GramTwo{Y'}{U'}\text{ and } \sigma(M') \in \Delta(S).
    \text{)}
  \end{aligned}
\]
Therefore, $(X,Y, U,V)$ (resp.\@ $(X\ V)$, $(Y\ U)$) and $(X',Y', U', V')$ (resp.\@ $(X'\ V')$, $(Y'\ U')$) build the same Gram matrix, i.e., $G = \GramFour XYUV = \GramFour{X'}{Y'}{U'}{V'}$ (resp.\@ $G = \GramTwo XY = \GramTwo{X'}{Y'}$, $H=\GramTwo UV =\GramTwo{U'}{V'}$).

{Without loss of generality, we may assume $X',Y', U', V'$ have no more rows than $X,Y, U,V$ respectively. By Theorem~\ref{th:gram_equal_exist_isometric}, there exists an isometric matrix $R$ (resp.\@ $R,S$) such that $RX'=X,$ $RY'=Y$, $RU'=U$, $RV'=V$ (resp.\@ $Y=SY'$, $U=SU'$), then $X'=R^{T} X$, $U'=R^{T} U$, $Y=RMR^{T} X$ and $V= R(M')^{T} R^{T} U$ (resp.\@ $Y=SM' R^{T} X$, $V=S(M')^{T} U$). Since $RM' R^{T} \in \mathcal{S}_S$ (resp.\@ $S M' R^{T} \in \mathcal{L}_S$), we have $(X,Y,U,V)$ is $\mathcal{S}_S$ -interpolable (resp.\@ $\mathcal{L}_S$-interpolable).}
\end{proof}
\end{theorem}
\addDJune{%
  Combined with our main results, if the sets of Gram matrices $\GM{\mathcal{M}}$ or $\GMM{\mathcal{M}}$ over the class of linear operators $\mathcal{M}$ are convex, then the Gram matrices inclusion tests provide interpolation conditions for $(X,Y,U,V)$ over the same class of linear operators.
}


\section{Interpolation conditions for $\LL{[\mu,L]}$} \label{sect:interpolation}

In this section, we propose interpolation conditions for the class $\LL{[\mu,L]}$ of linear operators whose singular values lie in the interval $[\mu,L]$. The idea is to exploit the polar decomposition to write any such operator as the product between a isometric transformation and a symmetric positive semidefinite linear operator whose eigenvalues also confined to $[\mu,L]$, for which we already have interpolation conditions (\cite[Theorem 7.3.11]{horn2012matrix} and \cite[Theorem 3.3]{bousselmi2024interpolation}).

\repDOct[Jhendrickx: is the meaning obvious due to the definition R? 
So it's R that's isometric, not $R^T$.]{%
The interpolation conditions for the class of unitary and isometric transformations are not new but we propose a \comN{new} formulation based on Definition \ref{def:R_mat_inter1}.
A linear operator $R$ is called isometric if $R^{T} R=I$. If, in addition, $R$ also satisfies $RR^T = I$, then $R$ is unitary. Note that an isometric operator need not have an isometric transpose: only when both $R^T R = I$ and $RR^T = I$ hold does $R$ becomes unitary.
\begin{definition}[Unitary and isometric interpolability] \label{def:unitary_isometric_inter}
Given two collections of data pairs $\{(x_i,y_i)\}_{i\in[N_1]}\subseteq\mathbb{R}^{n}\times\mathbb{R}^{m},$
$\{(u_j,v_j)\}_{j\in[N_2]}\subseteq\mathbb{R}^{m}\times\mathbb{R}^{n}$,
we say that the pairs are unitary‑interpolable (resp.\@ isometric-interpolable) if there exists a unitary (resp.\@ an isometric) linear operator $R$ such that
  \[
      \begin{cases}
          y_i = R   x_i & \forall i\in[N_1],
     \\   v_j = R^T u_j & \forall j\in[N_2].
      \end{cases}
  \]
\end{definition}
}{}

\begin{lemma}[Unitary interpolation conditions]\label{th:int_cond_unit}
The sets of pairs $\{(x_i,y_i)\}_{i\in [N_1]}$ and $\{(u_j,v_j)\}_{j\in [N_2]}$ are unitary-interpolable if, and only if,
\begin{equation} \label{cond:unit}
\begin{cases}
Y^{T} Y= X^{T} X,  \\
Y^{T} U = X^{T} V, \\
U^{T} U = V^{T} V .
\end{cases}
\end{equation}
\end{lemma}
\begin{proof}
We exploit the fact that the transpose of a unitary matrix is its inverse.
$\{(x_i,y_i)\}_{i\in [N_1]}$ and $\{(u_j,v_j)\}_{j\in [N_2]}$ are unitary-interpolable if and only if there exist an unitary matrix $R \in \R^{n \times n}$, such that $Y=RX$ and $U=R V$. It is equivalent to $\vrow{Y}{U}^{T} \vrow{Y}{U}  = \vrow{X}{V}^{T} \vrow{X}{V}$ by Theorem~\ref{th:gram_equal_exist_isometric}.
\end{proof}
Since the following fact on Gram matrices and isometric linear operator is invoked repeatedly in the subsequent proofs, we recall it here for clarity and readability.
\begin{lemma}[Theorem 7.3.11 of \cite{horn2012matrix}]\label{th:gram_equal_exist_isometric}
Let $N, m, n$ be positive integers with $m\le n$. Let $A \in \R^{m\times N}$ and $B \in \R^{n\times N}$. Then $A^{T} A=B^{T} B$ if and only if there is an isometric matrix $R \in \R^{n \times m}$ such that $B=RA$.
\end{lemma}
\begin{proposition}[Isometric interpolation conditions]\label{th:int_cond_isometric}
The sets of pairs $\{(x_i,y_i)\}_{i\in [N_1]}$ and $\{(u_j,v_j)\}_{j\in [N_2]}$ are isometric-interpolable if, and only if,
\begin{equation} \label{cond:isom}
\begin{cases}
Y^{T} Y= X^{T} X,  \\
Y^{T} U = X^{T} V, \\
U^{T} U \succeq V^{T} V.
\end{cases}
\end{equation}
\end{proposition}
\begin{proof}
Given an isometric matrix ${R}$, we can find a unitary matrix $\bar{R}$ such that ${R} =\bar{R}\vcol{I}{0}$. Therefore, the sets of pairs $\{(x_i,y_i)\}_{i\in [N_1]}$ and $\{(u_j,v_j)\}_{j\in [N_2]}$ are isometric-interpolable if, and only if,
\begin{alignat*}{3}
                & \exists R\text{ isometric},
                & \ \                                & Y=R X \text{ and } V = R^{T} U                                       \\
\Leftrightarrow\ & \exists \bar{R} \text{ unitary},
                &                                    & Y= \bar{R} \vcol{I}{0} X  \text{ and } V = \vrow{I}{0} \bar{R}^{T} U \\
\Leftrightarrow\ & \exists \bar{R} \text{ unitary}, Z,
                &                                    & Y=\bar{R}\vcol{X}{0} \text{ and } \vcol{V}{Z} = \bar{R}^{T} U        \\
%
\overset{\text{Th.\@ \text{\ref{th:int_cond_unit}}}}{\Leftrightarrow}\ 
                & \exists Z,
                &                                    & \vrow{Y}{U} ^{T} \vrow{Y}{U} =
\vrow{\vcol{X}{0}}{\vcol{V}{Z}} ^{T} \vrow{\vcol{X}{0}}{\vcol{V}{Z}}                                                        \\
\Leftrightarrow\ & \exists Z,
                &                                    & \begin{cases}
                                                       Y^{T} Y= X^{T} X,  \\
                                                       Y^{T} U = X^{T} V, \\
                                                       U^{T} U=V^{T} V + Z^{T} Z ,
                                                       \end{cases}                                               \\
\Leftrightarrow\ &
                &                                    & \begin{cases}
                                                       Y^{T} Y= X^{T} X,  \\
                                                       Y^{T} U = X^{T} V, \\
                                                       U^{T} U \succeq V^{T} V.
                                                       \end{cases}
\end{alignat*}
\end{proof}

\repDOct[JH: Maybe it's worth mentioning that x,y u,v being isometric interpolable does not imply that u,v x,y is. hance the asymmetry of conditions.]{%
The two interpolation conditions differ only in the third line of systems
 \eqref{cond:unit} and \eqref{cond:isom}. 
In the isometric case, the equality is replaced by an inequality.
As a result, the isometric interpolation condition 
becomes asymmetric: the quadruple $(X,Y,U,V)$ may be isometrically interpolable,
 while the reversed $(U,V,X,Y)$ need not be.
This asymmetry arises because the transpose of an isometric matrix $R$ is not necessarily isometric.}{%
}

\comN{For the sake of completeness, we recall the theorem on the polar decomposition.
\begin{theorem}[Polar decomposition, Theorem 7.3.1 of \cite{horn2012matrix}]\label{th:polar}
Let \comD{$M$ be} a matrix of size $m\times n$ with $m \ge n$. We have
\begin{equation}
M = {R} \comD{Q}
\end{equation}
where $\comD{Q} = (M^{T}M)^\frac12$ is a symmetric \comD{positive semidefinite} matrix whose eigenvalues are the singular values of $M$ and $R$ is an {isometric matrix}.
\end{theorem}
}

We denote the class of tall and wide linear operators as follows
\[
\Ltal{S}  = \LL{S} \cap \{M\in \mathbb{R}^{m \times n}  \mid m\ge n\} , \qquad
\Lwide{S}  = \LL{S} \cap \{M\in\mathbb{R}^{m \times n}  \mid m\le n\}.
\]
where $S \subseteq [0, +\infty)$.
\repDAug[]{These two classes of linear operators are often explicitly considered in the optimization context. For example, $\Ltal{S}$ appears in compressed sensing, while $\Lwide{S}$ arises in linearly constrained optimization.
The following proposition shows that the interpolation condition of $\Ltal{S}$ and $\Lwide{S}$ can be derived from each other.
\begin{proposition} \label{prop:Ltall_Lwide_swapping}
    The sets of pairs $\{(x_i,y_i)\}_{i\in [N_1]}$ and $\{(u_j,v_j)\}_{j\in [N_2]}$ are $\Lwide{S}$-interpolable if and only if the sets of pairs $\{(u_j,v_j)\}_{j\in [N_2]}$ and $\{(x_i,y_i)\}_{i\in [N_1]}$ are $\Ltal{S}$-interpolable.
\end{proposition}
\begin{proof}
    Observe that $M \in \Ltal{S}$ if and only if $M^T \in \Lwide{S}$, since the transpose does not change the singular spectrum.
    The result then follows directly by applying Definition~\ref{def:R_mat_inter1} to $M$ and $M^T$.
\end{proof}
}{} 
The new theorem on interpolation conditions for $\Ltal{[\mu,L]}$ \repDOct[]{and $\Lwide{[\mu,L]}$}{} is the following. \repDOct[]{Note that given the inputs and outputs $(X,Y,U,V)$, we know whether the linear operators are tall, wide or square and which class should be interpolated.
}{}
\begin{theorem}[$\LL{[\mu,L]}$-interpolation conditions]
\label{thm:L_muL_interpolation_conditon}
When $m\ge n$, the sets of pairs $\{(x_i,y_i)\}_{i\in [N_1]}$ and $\{(u_j,v_j)\}_{j\in [N_2]}$ are $\Ltal{[\mu,L]}$-interpolable if, and only if,
\begin{equation} \label{cond:Ltal}
\exists Z \in \R^{n\times N_1}, W\in \R^{n\times N_2},
\begin{cases}
Y^{T} Y=Z^{T} Z                                           \\
Y^{T} U = Z^{T} W                                         \\
U^{T} U \succeq W^{T} W                                   \\
\vrow{X}{W}^{T} \! \vrow{Z}{V} = \vrow{Z}{V}^{T} \! \vrow{X}{W} \\
\left(\vrow{Z}{V} - \mu\vrow{X}{W}\right)^{T} \! \left(L  \vrow{X}{W} - \vrow{Z}{V} \right)  \succeq 0.
\end{cases}    
\end{equation}
When $m=n$, the third inequality must be satisfied by equality, i.e., $U^{T} U = W^{T} W$.
\repDOct[jh quesiton: should we include a reference to proposition 3.5 in thm 3.6 to have an exhaustive formualtion?]{%
When $m \le n$, the sets of pairs $\{(x_i,y_i)\}_{i\in [N_1]}$ and $\{(u_j,v_j)\}_{j\in [N_2]}$ are $\Lwide{[\mu,L]}$-interpolable if, and only if, the same conditions~\eqref{cond:Ltal} hold after the following swaps: $X \leftrightarrow U$, $Y \leftrightarrow V$, $n \leftrightarrow m$, and $N_1 \leftrightarrow N_2$.
}{}
\end{theorem}
\begin{proof}
Consider the polar decomposition (Theorem~\ref{th:polar}) of the matrix $M=RQ$ where $Q \in \SSS{[\mu,L]}$, and $R$ is an isometric matrix.
\begin{alignat*}{2}
\exists M \in \LL{[\mu,L]}^{\mathrm{tall}} :
\left\{\begin{aligned}
       Y & = MX    \\
       V & = M^{T} U
       \end{aligned}\right.
 & \Leftrightarrow \exists R \text{ isometric, } Q \in \SSS{[\mu,L]} :
\begin{cases}
Y = R Q X, \\
V = Q R^{T}  U
\end{cases}                                                                                                                          \\    
 & \Leftrightarrow \exists R \text{ isometric, } Q \in \SSS{[\mu,L]},Z,W :
\begin{cases}
Y = R Z, \\
Z = Q X  \\
V = Q W  \\
W = R^{T}  U
\end{cases}                                                                                                                          \\    
 & \Leftrightarrow   \exists Z,W
\begin{cases}
(Z, Y, U, W) \text{ is isometric-interpolable,} \\
\vrow{\vrow{X}{W}}{\vrow{Z}{V}} \text{ is } \SSS{[\mu,L]} \text{-interpolable,}
\end{cases} \\    
\end{alignat*}%
\[ \overset{\text{Th.\@ \ref{th:int_cond_sym} and \ref{th:int_cond_isometric}}}{
\Leftrightarrow
}\ \exists Z, W
\begin{cases}
Y^{T} Y=Z^{T} Z                                           \\
Y^{T} U = Z^{T} W                                         \\
U^{T} U \succeq W^{T} W                                   \\
\vrow{X}{W}^{T} \vrow{Z}{V} = \vrow{Z}{V}^{T} \vrow{X}{W} \\
\left(\vrow{Z}{V} - \mu\vrow{X}{W}\right)^{T} \left(L  \vrow{X}{W} - \vrow{Z}{V} \right)  \succeq 0.
\end{cases}
\]
When $m=n$, $R$ is unitary and we can use Theorem~\ref{th:int_cond_unit}, therefore, the third inequality becomes the equality $U^{T} U = W^{T} W$. \\
\repDOct[]{%
For the case $m\le n$, by Proposition~\ref{prop:Ltall_Lwide_swapping}, the sets of pairs $\{(x_i,y_i)\}_{i\in [N_1]}$ and $\{(u_j,v_j)\}_{j\in [N_2]}$ are $\Lwide{[\mu,L]}$-interpolable if, and only if, the reversed pair $\{(u_j, v_j)\}_{j \in [N_2]}$ and $\{(x_i, y_i)\}_{i \in [N_1]}$ is $\Ltal{[\mu,L]}$-interpolable. Applying \eqref{cond:Ltal} to the latter statement yields the desired conditions.
}{}
\end{proof}

\repDOct[]{We can make several observations concerning the new interpolation conditions. 
The overall approach is based on the polar decomposition.
The first three lines of \eqref{cond:Ltal} are related to the isometric interpolability of $Y$, $U$ together with auxiliary intermediate variables $Z$ and $W$. The last two lines of \eqref{cond:Ltal} correspond to the $\SSS{[\mu,L]}$-interpolability of some augmented input-output matrices. This theorem extend the previous work \cite{bousselmi2024interpolation}}{}
Note that these interpolation conditions are convex constraints on Gram matrices of vectors $(X,Y,U,V)$ and the additional vectors $(Z,\comD{W})$, which is of the form \eqref{extended_gram}.

\section{Interpolation conditions for union of eigenvalues and singular values}\label{sect:union}
In this section, we propose an expression of interpolation conditions for \comN{symmetric} linear operators whose eigenvalues belong to a union of subsets of $\R$.
The idea is to sum the Gram matrices associated with each subset.
\begin{theorem}\label{th:union}
Let $S_1,S_2\subseteq \R$. 
$G\in \GM{\Delta(S_1 \cup S_2)}$ if, and only if, $G = G_1 + G_2$ for some $G_i \in \GM{\Delta(S_i)}$.
\repDAug[]{%
$(G,H)\in \GMM{\Delta(S_1 \cup S_2)}$ if, and only if, $G = G_1 + G_2$ and $H=H_1 + H_2$ for some $(G_i, H_i) \in \GMM{\Delta(S_i)}$, $i=1,2$.
}{}
\begin{proof} Let $G\in \mathcal{G}_{\Delta(S_1 \cup S_2)}$. Therefore, $G$ is associated with some matrix $Q$ with eigenvalues in $S_1 \cup S_2$ (Definition \ref{def:associated}) and 
using Lemma \ref{lem:unitary_symmetric}-1,
we can choose $Q$ with the a block-diagonal form
\repDAug[]{%
\begin{align*}
& G\in \GM{\Delta(S_1 \cup S_2)} 
\\ & \Leftrightarrow G=\GramFour{\vcol{X_1}{X_2}}{\vcol{Y_1}{Y_2}}{\vcol{U_1}{U_2}}{\vcol{V_1}{V_2}} \text{ for some } X_i,\ Y_i,\ U_i,\ V_i,\ i=1,2,
\\ & \quad \text{ and } \exists Q_1, Q_2 : \vcol{Y_1}{Y_2} = \vmat{Q_1}{0}{0}{Q_2} \vcol{X_1}{X_2}, \vcol{V_1}{V_2} = \vmat{Q_1}{0}{0}{Q_2}^T \vcol{U_1}{U_2}, \  \comD{\lambda(Q_i) \in \Delta(S_i)} 
\\ & \Leftrightarrow G = \GramFour{X_1}{Y_1}{U_1}{V_1} + \GramFour{X_2}{Y_2}{U_2}{V_2} \text{ for some } X_i,\ Y_i,\ U_i,\ V_i,\ i=1,2,
\\ & \quad \text{ and } \exists Q_1, Q_2 :  Y_i = Q_i X_i,\ V_i = Q_i^T U_i, \text{ and }   \lambda(Q_i) \in \Delta(S_i)
\\ & \Leftrightarrow G = G_1 + G_2 \text{ for some } G_i\in \GM{\Delta(S_i)},\ i=1,\ 2
\end{align*}
For the second statement,
 let $(G,H) \in \GMM{\Delta(S_1 \cup S_2)}$. Therefore, $(G,H)$ is associated with some matrix $Q$ with singular values in $S_1 \cup S_2$. Using Lemma \ref{lem:unitary}-2, similarly,
\begin{align*}
   & (G,H) \in \GMM{\Delta(S_1 \cup S_2)} 
\\ & \Leftrightarrow G=\GramTwo{\vcol{X_1}{X_2}}{\vcol{V_1}{V_2}}, H=\GramTwo{\vcol{Y_1}{Y_2}}{\vcol{U_1}{U_2}}, \text{ for some } X_i,\ Y_i,\ U_i,\ V_i,
\\ & \quad \text{ and } \exists Q_1, Q_2 : \vcol{Y_1}{Y_2} = \vmat{Q_1}{0}{0}{Q_2} \vcol{X_1}{X_2}, \vcol{V_1}{V_2} = \vmat{Q_1}{0}{0}{Q_2}^T \vcol{U_1}{U_2}, \ \comD{\lambda(Q_i) \in \Delta(S_i)}
\\ & \Leftrightarrow G = \GramTwo{X_1}{V_1} + \GramTwo{X_2}{V_2},\  H = \GramTwo{Y_1}{U_1} + \GramTwo{Y_2}{U_2},
\\ & \quad \ \text{for some } X_i,\ Y_i,\ U_i,\ V_i, i=1,2 \text{ and } \exists Q_1, Q_2:
\\ & \quad \ Y_i = Q_i X_i, \ V_i = Q_i^T U_i,\comD{\lambda(Q_i) \in \Delta(S_i)}, \ i=1, 2
\\ & \Leftrightarrow G = G_1 + G_2 \text{ for some } G_i\in \GM{\Delta(S_i)}, \ i=1,2.
\end{align*}%
}{}
\end{proof}
\end{theorem}
We already know interpolations for symmetric linear operators with eigenvalues on a given interval on $\mathbb{R}$,
\repDAug{%
and general linear operators with singular values on a given interval on $[0,+\infty)$.
}{}
Using the above theorem, Theorem~\ref{th:union_XY} and Theorem~\ref{th:int_cond_non_sym}, we can develop interpolation conditions for symmetric linear operators with eigenvalues on any (finite) union of intervals on $\mathbb{R}$
\repDAug{and interpolation conditions for general linear operators with singular values on any (finite) union of intervals on $[0,+\infty)$}{}.

For example, this theorem allows to \comN{develop} interpolation conditions for symmetric linear operators with eigenvalues equal to either $-\mu$ or $\mu$.
\begin{corollary}\label{cor:interp_mumu}
Let $\mu \in \R$. \repDOctui[]{$(X,Y)$ is $\SSS{\{-\mu,\mu\}}$-interpolable}{$(X,Y) \in \mathcal{S}_{\{ -\mu, \mu \}}$} if, and only if,

\begin{equation} \label{int_mu_-mu}
\begin{cases}
\repDOct[]{%
    X^T Y = Y^T X,
\\  \mu^2 X^T X = Y^T Y.
}{}
\end{cases}
\end{equation}

\begin{proof}
\repDOct[]{%
The equalities in \eqref{int_mu_-mu} is equivalent to the statement that the Gram matrix $G=\GramTwo XY$ has the block form
\begin{equation}\label{int_mu_-mu_gram}
G = \vmat{A}{B}{B}{\mu^2 A} \succeq 0, \text{ for some } A \succeq 0, \ B=B^T.
\end{equation}
Indeed, the first condition in \eqref{int_mu_-mu} forces the off‑diagonal blocks to be symmetric, while the second condition makes the lower‑right block equal to the upper-left block. Setting \(A:=X^T X\) and \(B:=X^T Y\) yields \eqref{int_mu_-mu_gram}.
}{}

\repDOctui[]{Assume that $(X,Y)$ is $\SSS{\{-\mu,\mu\}}$-interpolable.}{} By Theorems \ref{th:union} and \ref{th:union_XY}, we have $G \in \GM{\Delta(\{\mu\} \cup \{ -\mu\}  )}$ if, and only if, $G = G_1 + G_2$ with $G_1 \in \GM{\Delta(\{ \mu \} )}$ and $G_2 \in \GM{\Delta( \{ -\mu \} )}$. And, \repDOctui[]{by Theorem~\ref{th:int_cond_sym},}{} one can show that \repDOctui[]{for any $\alpha \in \R$,}{} 
\begin{equation}\label{int_alpha}
G_i \in \GM{\Delta( \{ \alpha \} )}
\text{ if, and only if }
G_i = \vmat{A}{\alpha A}{\alpha A}{\alpha^2 A}, \ A \succeq 0.
\end{equation}
\repDOctui[]{Applying \eqref{int_alpha} with $\alpha=-\mu$ and $\alpha=\mu$ yields}{%
}
\begin{align}
G = G_1 + G_2 = \vmat{A_1}{\mu A_1}{\mu A_1}{\mu^2 A_1} + \vmat{A_2}{-\mu A_2}{-\mu A_2}{\mu^2 A_2} = \vmat{A_1+A_2}{\mu(A_1-A_2)}{\mu(A_1-A_2)}{\mu^2 (A_1+A_2)}
\end{align}
for some $A_1$ and $A_2$. Hence $G$ \repDOct[]{can be written as \eqref{int_mu_-mu_gram}}{} with $A = A_1 + A_2$ and $B=\mu(A_1-A_2)$.

On the other hand, \repDOctui[]{assume that $G$ can be written as in \eqref{int_mu_-mu_gram}.}{%
}
Define $A_1 = \frac12 \left( A+\frac{B}{\mu} \right)$ and $A_2 = \frac12 \left( A-\frac{B}{\mu} \right)$.
\repDOct[]{Because $G\succeq 0$, the matrices $A_1$, $A_2$ are positive semidefinite since for any vector $x$, $\vcol{x}{\pm \frac{1}{\mu} x}^T \! \vmat{A}{B}{B}{\mu^2 A}  \vcol{x}{\pm \frac{1}{\mu} x} = x^T A_1 x \text{ or } x^T A_2 x \ge 0$.}{}
Using these matrices, we have $G=G_1+G_2$ with $G_1 = \vmat{A_1}{\mu A_1}{\mu A_1}{\mu^2 A_1}$ and $G_2 = \vmat{A_2}{-\mu A_2}{-\mu A_2}{\mu^2 A_2}$.
\repDOctui[]{By \eqref{int_alpha}, each summand belongs to \(\GM{\Delta(\{\mu\})}\) and
\(\GM{\Delta(\{-\mu\})}\), respectively. Hence \(G\in
\GM{\Delta(\{-\mu,\mu\})}\), which means that \((X,Y)\) is
\(\mathcal{S}_{\{-\mu,\mu\}}\)-interpolable.
}{}
\end{proof}
\end{corollary}
\repDOct[jh: make the difference more explicitly by re-writing them in the same way?
$G = (A B; B C)$ with \\
(translating the conditions) $mu B + mu^2 A - C - mu B >=0$ i.e. $mu^2 A >= C$ while in (23) there is an equality]{%
The interpolation condition for symmetric linear operators with eigenvalues in $\{-\mu, \mu\}$ (Corollary~\ref{cor:interp_mumu}) differs from the corresponding condition for $[-\mu, \mu]$ (Theorem~\ref{th:int_cond_sym}), which is given by
\[
\begin{cases}
X^T Y = Y^T X, \\
\mu^2 X^T X \succeq Y^T Y.
\end{cases}
\]
In the former case, the second condition becomes an equality.
}{%
}


\section{Application to Performance Estimation and worst-case analysis}\label{sect:applications}
The interpolation conditions derived in Theorems \ref{thm:L_muL_interpolation_conditon}, \ref{th:union} and \ref{th:union_XY}
allow to analyze the worst-case performance of optimization methods with PEP on problems of the forms $\min_x g(Mx)$ or $\min_x f(x) + g(Mx)$ where $M \in \LL{I}$ or $M \in \SSS{I}$ where $I$ is any subset of $\R.$ We propose to analyze the Gradient method on the first problem and the Chambolle-Pock method on the second problem.

\subsection{Worst-case performance of Gradient Method on smooth convex functions}
The exact convergence rate of Gradient Method (GM) on smooth convex functions is an active research area \cite{nesterov2018lectures,drori2014performance,taylor2017smooth,rotaru2024exact}. The work \cite{bousselmi2024interpolation}
considered the problem of the form
\begin{equation}
\min_x F(x) \triangleq g(Mx)
\end{equation}
where $g$ is $\mu_g$-strongly convex and $L_g$-smooth and $M\in \LL{[\mu_M,L_M]}$ (we say that $F \in \mathcal{C}_{\mu_g,L_g}^{\mu_M,L_M}$) or $M \in \SSS{[\mu_M,L_M]}$ (we say that $F\in \mathcal{D}_{\mu_g,L_g}^{\mu_M,L_M}$). They consider GM with fixed stepsize $h/L$
\begin{equation}
x_{k+1} = x_k - \frac{h}{L} \nabla F(x) = x_k - \frac{h}{L} M^T \nabla g(Mx) \tag{GM}
\end{equation}
where $L$ is the Lipschitz constant of the gradient of $F$. They characterized the worst-case behavior of GM on the classes $\mathcal{C}_{\mu_g,L_g}^{0,L_M}$ and $\mathcal{D}_{\mu_g,L_g}^{\mu_M,L_M}$ with $w(\mathcal{F},h/L)$ being the worst-case behavior of GM with step size $h/L$ on the class $\mathcal{F}$ with respect to the criterion $F(x_N) - F(x^*)$ and bounded initial distance $\|x_0 - x^*\| \leq 1$ on \cite[Conjecture 4.2]{bousselmi2024interpolation}.

Prior to this work, computing the tight performance on $\mathcal{C}_{\mu_g,L_g}^{\mu_M,L_M}$ for $\mu_M >0$ is impossible due to the unknown interpolation conditions for $\LL{[\mu_M,L_M]}$. Now that we have them (Theorem \ref{thm:L_muL_interpolation_conditon}), we can extend their analysis of \cite{bousselmi2024interpolation} to $\mathcal{C}_{\mu_g,L_g}^{\mu_M,L_M}$. 
\subsubsection{Performance on linear operators with non-zero minimal singular value}
We compute the worst-case performance of GM when the linear operater $M$ is with maximal and nonzero minimal bound on the singular values. 
We observed that both in the tall and wide case, the worst-case performance of GM is exactly the same as the worst-case performance of GM on $\mathcal{D}_{\mu_g,L_g}^{\mu_M,L_M}$, namely, the symmetry of $M$ has no impact on the worst-case performance.

\subsubsection{Performance on linear operators with eigenvalues on a union of intervals}
We compute the worst-case performance of GM when the linear operator $M$ has eigenvalues in a union of intervals. Figure \ref{fig:union} compares the worst-case linear operator of GM when varying the step size $h$ when $M \in \SSS{[0.6, 1]}$ (red dots) and $M \in \SSS{[0.6, 0.65] \cup [0.95, 1]}$ (blue dots) which appears to be both scalar. In other words, either we include the whole interval $[0.6, 1]$, or we exclude the sub-interval $[0.65, 0.95]$.

\begin{figure}[H]
\centering
\includegraphics[width=\linewidth]{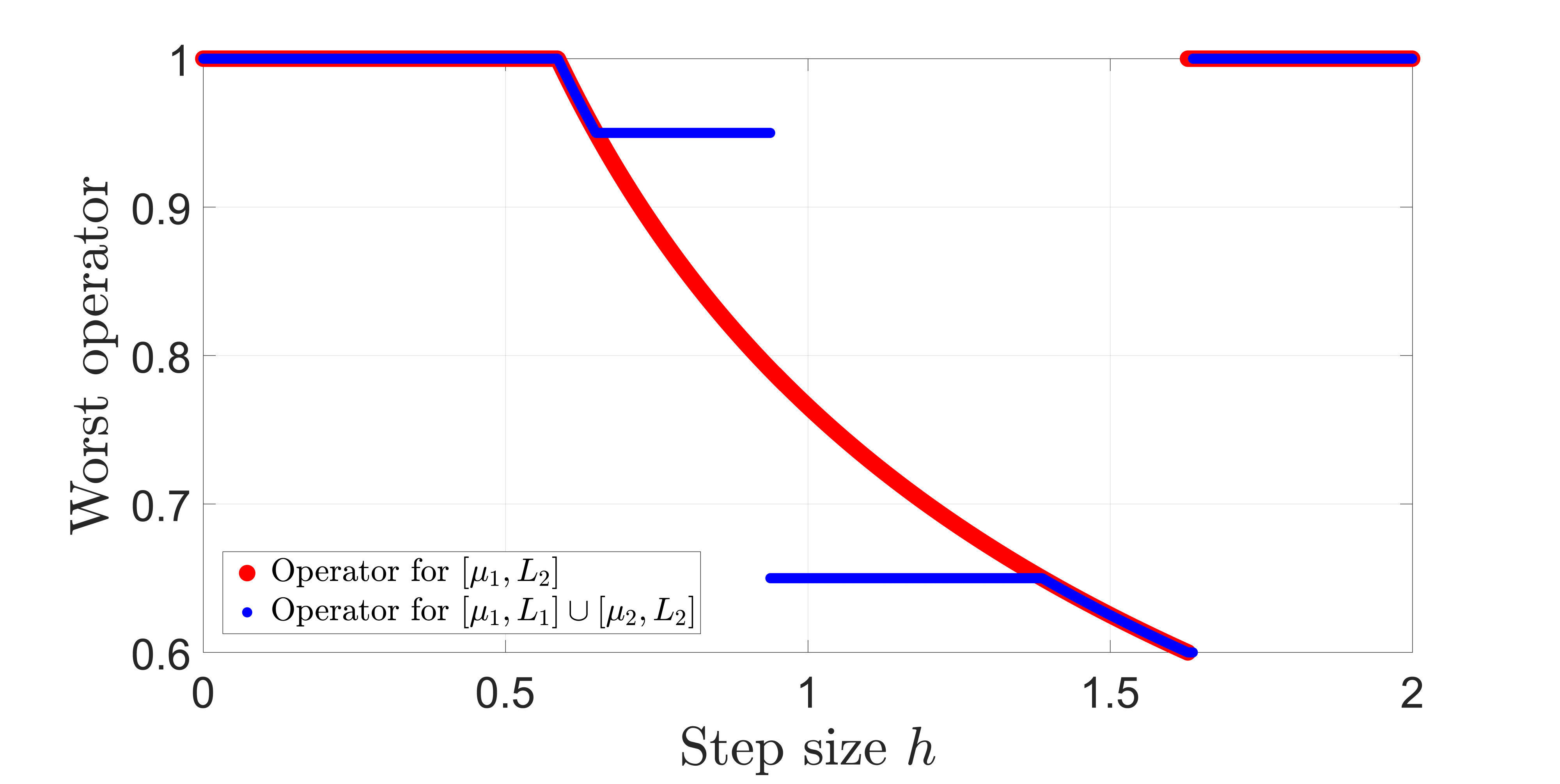}
\caption{Comparison of the worst-case linear operators when allowing the whole interval (red curve) or only the union of two sub-intervals (blue curve).}
\label{fig:union}
\end{figure}

Interestingly, the worst-case linear operator remains a scalar. This could not be observed without our new interpolation conditions.

\subsection{Chambolle-Pock Method for composite optimization problem}
Another application of our interpolation conditions is the analysis of the Chambolle-Pock method (CP) \cite{chambolle2011first} on the composite optimization problem
\[ \min_x f(x) + g(Mx) \]
where $f,g$ are proximable convex functions. The iterations are the following, with parameters $\tau, \sigma > 0$
\[ \left\{\begin{array}{l}
x_{i+1}=\operatorname{prox}_{\tau f}\left(x_i-\tau M^T u_i\right), \\
u_{i+1}=\operatorname{prox}_{\sigma g^*}\left(u_i+\sigma M\left(2 x_{i+1}-x_i\right)\right).
\end{array}\right. \]
The work \cite{bousselmi2024interpolation} also analyzed the worst-case behavior of CP on different settings with PEP. More precisely, they improved \cite[Theorem 1]{chambolle2016ergodic} the convergence rate of the primal-dual gap of the average iterations produced by CP (see \cite[Theorem 4.6]{bousselmi2024interpolation}) when $M \in \LL{[0,L_M]}$ for some $L_M > 0$. Note that this classical result requires a bound on
\begin{equation}
\frac{1}{\tau} \lVert x-x_0 \rVert^2 + \frac{1}{\sigma} \lVert u-u_0 \rVert^2 - 2(u-u_0)^T M (x-x_0)
\end{equation}
and \cite{bousselmi2024interpolation} also considered the bound
\begin{equation}
\frac{1}{\tau} \lVert x-x_0 \rVert^2 + \frac{1}{\sigma} \lVert u-u_0 \rVert^2.
\end{equation}
Thanks to our new interpolation conditions (Theorem \ref{th:int_cond_non_sym}), we can observe the impact of bounding the minimal singular values of the linear operators $M$. 
We found that considering the lower bound of singular value on $M$ does not improve the convergence rate in both settings.


\section{Conclusion}
Interpolation conditions for linear operators and PEP allow to analyze the performance of optimization methods on problems of the form $\min_x g(Mx)$ and $\min_x f(x) + g(Mx)$ and other problems. We used to have interpolation conditions for linear operators and their transpose characterized by singular values in $[0,L]$ and for symmetric linear operators characterized by eigenvalues in $[\mu,L]$.

In this work, we developed interpolation conditions for the class of linear operators with singular values between $\mu$ and $L$ and for symmetric linear operators characterized by any finite union of intervals for eigenvalues. The reasoning to obtain interpolation conditions for linear operators with singular values in a finite union of intervals is the same but we did not present it. We illustrated the application of our new interpolation conditions for the worst-case analysis of the gradient method and Chambolle-Pock methods via the Performance Estimation Problem framework.

We also showed that we can only characterize classes of linear operators \comN{whose} singular values are in a given subset of $\R$ and \comN{classes of symmetric linear operators whose eigenvalues are in a given subset of $\R$ through Gram representation}.

Therefore, there are the following open questions:
\begin{itemize}
 \item Can we develop interpolation conditions for linear operators whose singular values or eigenvalues belong to arbitrary sets that cannot be expressed as a finite union of intervals?
 \item Can we develop interpolation conditions for general square matrices, and which properties could we require?
 \item Is there another convex representation of classes of linear operators that does not rely on Gram representations?
\end{itemize}

\appendix


\section{Proofs of Theorem \ref{th:limit}}\label{app:proof_limit}

\begin{lemma}\label{lem:cone_only}
Let a set of spectrum $\Lambda$. $\GMM{\Lambda}$ and $\GM{\Lambda}$ are cones, namely,
\begin{align}
c (G,H) \in \GMM \Lambda &  & \forall c \geq 0, \forall (G,H) \in \GMM{\Lambda} \\
c G \in \GM \Lambda      &  & \forall c \geq 0, \forall G \in \GM{\Lambda}.
\end{align}
\begin{proof} For $c>0$, we have
\begin{alignat*}{3}
(G,H) \in \GMM{\Lambda} & \Rightarrow \exists X,U,M:  & \  & (G,H) = \Par{\GramTwo{X}{M^{T}U}, \GramTwo{MX}{U}}                                      \\
                        &                             &    & \text{and } \sigma(M) \in \Lambda                                                     \\
                        & \Rightarrow  \exists X,U,M: &    & c (G,H) = \Par{\GramTwo{\sqrt{c}X}{M^{T} \sqrt{c} U}, \GramTwo{\sqrt{c}MX}{ \sqrt{c}U}} \\
                        &                             &    & \text{and } \sigma(M) \in \Lambda                                                     \\
                        & \Rightarrow c (G,H) \in     &    & \GM{\Lambda}.
\end{alignat*}
and
\begin{align*}
G\in \GM{\Lambda} & \Rightarrow \exists X,U,\text{symmetric }Q: G = \GramFour{X}{QX}{U}{Q^{T}U} \text{ and } \lambda(Q) \in \Lambda                                       \\
                  & \Rightarrow \exists X,U,\text{symmetric }Q: c G = \GramFour{\sqrt{c}X}{Q \sqrt{c} X}{\sqrt{c}U}{Q^{T} \sqrt{c} U} \text{ and } \lambda(Q) \in \Lambda \\
                  & \Rightarrow c G \in \GM{\Lambda}.
\end{align*}

\end{proof}
\end{lemma}
\comN{The following lemma formalizes and exploits the fact that when a Gram matrix is associated with a linear operator $M_1$ and vectors $X,Y,U,V$, we can always augment the linear operator $M_1$ as $\vmat{M_1}{0}{0}{M_2}$ without modifying the Gram matrix by augmenting the vectors with zeros, namely, $\vcol{X}{0},\vcol{Y}{0},\vcol{U}{0},\vcol{V}{0}$. Therefore, when a Gram matrix is associated with some spectrum $\lambda_1$ it is also associated with all larger spectrum $\lambda_2$ such that $\lambda_1 \subseteq \lambda_2$.}
\begin{lemma}[Extension of spectrums]\label{lem:extens}
Let $\lambda_1$ and $\lambda_2$ be two spectrums . If $\lambda_1 \subseteq \lambda_2$ then $\GMM{\{\lambda_1\}} \subseteq \GMM{\{\lambda_2\}}$ and  $\mathcal{G}_{\{\lambda_1\}} \subseteq \mathcal{G}_{\{\lambda_2\}}$.
\begin{proof}
Let $G\in \GM{\{\lambda_1 \}}$ (resp.\@ $(G,H)\in \GMM{\{\lambda_1\}}$). We have
\begin{align*}
            & G = \GramFour XYUV \text{ (resp.\@ $(G,H) = \Par{\GramTwo XV, \GramTwo YU}$)}                                                                                              \\
            & \text{with }
\begin{cases}
Y & =M_1X    \\
V & =M_1^{T} U
\end{cases}
\text{ and } \lambda(M_1) = \lambda_1 \text{ (resp.\@ $\sigma(M_1) = \lambda_1$)}                                                                                                        \\
  \Rightarrow & G = \GramFour{\vcol{X}{0}}{\vcol{Y}{0}}{\vcol{U}{0}}{\vcol{V}{0}}\\& \text{ (resp.\@ $(G,H) = \Par{\GramTwo{\vcol{X}{0}}{\vcol{V}{0}}, \GramTwo{\vcol{Y}{0}}{\vcol{U}{0}}}$)} \\
            & \text{with } \begin{cases}
                           \vcol{Y}{0}=\vmat{M_1}{0}{0}{M_2}\vcol{X}{0} \\
                           \vcol{V}{0}=\vmat{M_1^{T}}{0}{0}{M_2^{T}}\vcol{U}{0}
                           \end{cases} \!\!\! \text{ and }
\begin{cases}
\lambda(M_1) = \lambda_1              \text{\quad \quad\ \  (resp.\@ $\sigma(M_1) = \lambda_1$)}             \\
\lambda(M_2) = \lambda_2 - \lambda_1  \ \text{(resp.\@ $\sigma(M_2) = \lambda_2 - \lambda_1$)}
\end{cases}                                                                                            \\
\Rightarrow & G \in \GM{\{\lambda_1 + (\lambda_2-\lambda_1)\}} =  \GM{\{\lambda_2\}} \text{ (resp.\@ $(G,H) \in \GMM{\{\lambda_1 + (\lambda_2-\lambda_1)\}} =  \GMM{\{\lambda_2\}} $)}
\end{align*}
since $ \lambda_1 \subseteq \lambda_2$, we have $\lambda_1 + (\lambda_2 - \lambda_1) = \lambda_2$.
\end{proof}
\end{lemma}
This property also holds for sets of spectrums.
\begin{corollary}\label{cor:5}
Let $\Lambda_1$ and $\Lambda_2$ be sets of spectrums. If for every $\lambda_1 \in \Lambda_1$ there exists $\lambda_2 \in \Lambda_2$ such that $\lambda_1 \subseteq \lambda_2$ then $ \GMM{\Lambda_1} \subseteq \GMM{\Lambda_2} $ and $\GM{\Lambda_1} \subseteq \GM{\Lambda_2}$.
\end{corollary}

\begin{theorem}[Theorem \ref{th:limit}]
Let $\mathcal{M}$ be a set of linear operators (resp.\@ symmetric matrices). If $\GMM{\mathcal{M}}$ (resp.\@ $\GM{\mathcal{M}}$)
is convex, then
\begin{equation}
\GMM{\mathcal{M}} = \GMM{\Delta(S)} \text{ (resp.\@ $\GM{\mathcal{M}} = \GM{\Delta(T)}$)}
\end{equation}
for some $S\subseteq \R$. 
\begin{proof}
\comN{By Corollary \ref{cor:spectr_chara}, we have $\GMM{\mathcal{M}} = \GMM{\Lambda}$ (resp.\@ $\GM{\mathcal{M}} = \GM{\Lambda}$) where $\Lambda = \{ \sigma(M) \,:\, M \in \mathcal{M} \}$ (resp.\@ $\Lambda = \{ \lambda(M) \,:\, M \in \mathcal{M} \}$).}
We show that $S$ (resp.\@ $T$) is the set of all values appearing in the spectrums of $\Lambda$, namely,  $S_\Lambda  \triangleq \{ \mu : \exists \lambda \in \Lambda \text{ s.t. } \mu \in \lambda  \}$.  Recall that
\begin{align}
\Delta(S_\Lambda) & = \{ \lambda : \forall a\in \lambda,~a\in S_\Lambda \} =\{ \lambda : \forall a\in \lambda,~a \in \lambda_a',~ \lambda_a' \in \Lambda \}
\end{align}
\comN{is the set of spectrums $\lambda$ for which each element $a$ belongs to some spectrum $\lambda_a'$ of $\Lambda$.} \\

\noindent (\textit{Inclusions $\GMM{\Lambda} \subseteq \GMM{\Delta(S_\Lambda)}$ and $\GM{\Lambda} \subseteq \GM{\Delta(S_\Lambda)}$}) \comN{Let a spectrum $\lambda\in \Lambda$, then clearly all the elements of this spectrum belong to some spectrum of $\Lambda$ and therefore $\lambda \in \Delta(S_\Lambda)$. Thus, we have} $ \Lambda \subseteq \Delta(S_\Lambda)$
and therefore $\GMM{\Lambda} \subseteq \GMM{\Delta(S_\Lambda)}$ and $\GM{\Lambda} \subseteq \GM{\Delta(S_\Lambda)}$ by Corollary \ref{cor:5}. \\
\noindent (\textit{Inclusions $\GMM{\Delta(S_\Lambda)} \subseteq \GMM\Lambda$ and $\GM{\Delta(S_\Lambda)} \subseteq \GM\Lambda$}) First, we show that $\mathcal{G}_{\{\lambda\}} \subseteq \GM{\Lambda}$ (resp.\@ $\GMM{\{\lambda\}} \subseteq \GMM{\Lambda}$) $\forall \lambda \in \Delta(S_\Lambda)$. \comN{Recall that $\mathcal{G}_{\{\lambda\}}$ (resp.\@ $\GMM{\{\lambda\}}$) is the set of Gram matrices associated with linear operators whose eigenvalue (resp.\@ singular value) spectrum is exactly $\lambda$}.
Let $\lambda \in \Delta(S_\Lambda)$ and $G\in \GM{{\{\lambda\}}}$ (resp.\@ $(G,H)\in \GMM{\{\lambda\}}$). We have $G=\GramFour XYUV$ (resp.\@ $(G,H)=(\GramTwo{X}{V},\GramTwo{Y}{U})$) with $Y = MX$, $V = M^{T} U$ and $a \in \lambda_a',~ \lambda_a' \in \Lambda, ~\forall a \in \lambda(M)$ (resp.\@ $\forall a \in \sigma(M)$). W.l.o.g, let $M$ of dimension $m \times d$ with $d\leq m$ in the rectangle case and of dimension $d \times d$ in the symmetric case. By Lemma \ref{lem:unitary} and eigenvalue (resp.\@ singular value) decomposition of $M$, we can choose $M$ diagonal with its $d$ eigenvalues (resp.\@ singular values) $m_i$ on its diagonal. Therefore, we have
\begin{equation}
(X,Y,U,V) = \Par{\begin{pmatrix}
X_1 \\ \vdots \\ X_d
\end{pmatrix}, \begin{pmatrix}
V_1 \\ \vdots \\ V_{d}
\end{pmatrix}, \begin{pmatrix}
Y_1 \\ \vdots \\ Y_m
\end{pmatrix}, \begin{pmatrix}
U_1 \\ \vdots \\ U_{m}
\end{pmatrix}}
\end{equation}
and
\begin{align*}
 & \begin{pmatrix}
   Y_1 \\ \vdots \\ Y_m
   \end{pmatrix} = \begin{pmatrix}
                   m_1 &        &     \\
                       & \ddots &     \\
                       &        & m_d \\
                   -   & E      & -
                   \end{pmatrix}\begin{pmatrix}
                                X_1 \\ \vdots \\ X_d
                                \end{pmatrix},
\begin{pmatrix}
V_1 \\ \vdots \\ V_d
\end{pmatrix} = \begin{pmatrix}
                m_1 &        &     & |   \\
                    & \ddots &     & E^{T} \\
                    &        & m_d & |
                \end{pmatrix}\begin{pmatrix}
                             U_1 \\ \vdots \\ U_m
                             \end{pmatrix},
\end{align*}
where $m_i \in \lambda_i', \lambda_i' \in \Lambda$, and $E=0$ with the appropriate size. It is equivalent to
\[
\Leftrightarrow Y_i = m_i X_i,  V_i = m_i U_i ~\forall i=1,\ldots,d \text{ and } Y_i = 0 ~\forall i=d+1,\ldots,m, m_i \in \lambda_i', \lambda_i' \in \Lambda
\]
$
\Leftrightarrow \vcol{Y_i}{0} = \overbrace{\vmat{m_i}{0}{0}{\bar{M}_i}}^{M_i} \vcol{X_i}{0},~\vcol{V_i}{0} = \overbrace{\vmat{m_i}{0}{0}{\bar{M}_i^{T}}}^{M_i^{T}} \vcol{U_i}{0},  \lambda(M_i) = \lambda_i' \text{ (resp.\@ }\sigma(M_i) = \lambda_i'\text{)}, \forall \lambda_i' \in \Lambda
$
where $\bar{M}_i$ completes the spectrum of $M_i$ such that the spectrum of $M_i$ is exactly $\lambda_i'$. Let $G_i = \GramFour{\vcol{X_i}{0}}{\vcol{Y_i}{0}}{\vcol{U_i}{0}}{\vcol{V_i}{0}} \in \GM{\Lambda}$ (resp.\@ $(G_i,H_i) = (\GramTwo{\vcol{X_i}{0}}{\vcol{V_i}{0}},\GramTwo{\vcol{Y_i}{0}}{\vcol{U_i}{0}} \in \GMM{\Lambda}$). We have $G = \sum_{i=1}^d G_i = \sum_{i=1}^d \frac{1}{d} (dG_i)$ (resp.\@ $(G,H) = \sum_{i=1}^d (G_i,H_i) = \sum_{i=1}^d \frac{1}{d} (d(G_i,H_i))$ since $Y_i=0$ for $i= d+1, \ldots, m$.) and $dG_i \in \GM{\Lambda}$ (resp.\@ $d(G_i,H_i) \in \GMM{\Lambda}$) by Lemma \ref{lem:cone_only}. And by convexity of $\GM{\Lambda}$ (resp.\@ $\GMM{\Lambda}$), we have $G = \sum_{i=1}^d \frac{1}{d} (dG_i) \in \GM{\Lambda}$ (resp.\@ $(G,H) = \sum_{i=1}^d \frac{1}{d} (d(G_i,H_i)) \in \GMM{\Lambda}$).
\\
\noindent Finally, let $G \in \GM{\Delta(S_\Lambda)}$ (resp.\@ $(G,H) \in \GMM{\Delta(S_\Lambda)}$).  By definition of $\Delta(S_\Lambda)$, there is a $\lambda \in \Delta(S_\Lambda)$ such that $G \in \GM {\{\lambda\}}$ (resp.\@ $(G,H) \in \GMM{\{\lambda\}}$) and therefore $G\in \GM{\Lambda}$ (resp.\@ $(G,H) \in \GMM{\Lambda}$).
\end{proof}
\end{theorem}

\nocite{chambolle2011first}

\bmhead{Acknowledgements}
Nizar Bousselmi is an FRIA grantee of the Fonds de la Recherche Scientifique - FNRS.
Zhicheng Deng is sponsored by China Scholarship Council.

\section*{Delcarations}

\bmhead{Data Availibility Statement}
The datasets generated during the current study are available from the corresponding author on reasonable request.

\bmhead{Conflict of interest} This research was conducted in full compliance with ethical standards, with no Conflict
of interest, no involvement of human participants or animals, and no requirement for informed consent.

\bibliography{sn-bibliography}

\end{document}